\documentclass{amsart} 
\usepackage{latexsym,amssymb,graphicx,amscd}

\newtheorem{thm}{Theorem}

\newtheorem{prop}[thm]{Proposition}

\newcommand{\BZ}{{\mathbb{Z}}}

\newcommand{\Si}{{\Sigma}}

\begin{document}
\title{Remarks on Congruence of $3$--manifolds}.

\author{ Patrick M. Gilmer}
\address{Department of Mathematics\\
Louisiana State University\\
Baton Rouge, LA 70803\\
USA}
\email{gilmer@math.lsu.edu}
\thanks{partially supported by NSF-DMS-0604580}
\urladdr{www.math.lsu.edu/\textasciitilde gilmer/}

\dedicatory{Dedicated to Lou Kauffman on the occasion of his 60th birthday}

\date{November 15, 2007}

\begin{abstract}  We give two proofs that the $3$--torus is not weakly $d$--congruent to $\#^3 S^1 \times S^2$, if $d>2$. We study how cohomology ring structure relates to weak congruence. We give an example of three 3--manifolds which are weakly $5$--congruent but are not $5$--congruent.  \end{abstract}

\maketitle

Let $d$ be an integer greater than one. In \cite{G}, we considered two equivalence relations generated by restricted surgeries on oriented closed $3$--manifolds. Weak type--$d$ surgery is $q/ds$ Dehn surgery along a simple closed curve. Here $s$ and $q$ (which must be relatively prime to $d$ and $s$) may vary but $d$ is held fixed. The label $q/ds$ indicating which surgery is given with respect to some  meridional and a longitudal pair  on the boundary of a solid torus neighborhood of the surgery curve. A meridian bounds a disk in the solid torus which meets the surgery curve transversely in one point, and a longitude meets the meridian transversely in one point in the boundary torus. The set of surgeries described as weak type--$d$ surgeries does not depend on the choice of meridional and a longitudinal pair. If $q= \pm 1 \pmod{d}$, we say the surgery is type--$d$ surgery. This concept is also independent of the choice of meridian and longitude.

The equivalence relation on the set of closed oriented $3$--manifolds generated by weak type--$d$ surgery is called weak $d$--congruence. The equivalence relation  generated by  type--$d$ surgery is called  $d$--congruence. 

The equivalence relation $d$--congruence is coaser  \cite{G} than an  equivalence relation which was first considered by Lackenby \cite{L} : congruence modulo $d$. It is not known that $d$--congruence is strickly coaser than congruence modulo $d$, but this seems likely. 
The notion of $d$-congruence of 3-manifolds is closely related to the notion of $t_d$-move (now called $d$-move) equivalence of links  \cite[remark before proof of Theorem p.639]{Pr1}: A $d$-move between links implies that there is
$1/d$ Dehn surgery relating the double branched covers of $S^3$ along  the links.  Similarly, weak $d$-equivalence of 3-manifolds is closely related to rational move equivalence of links as analyzed in 
\cite[footnote 5]{P2}, \cite[footnote 22]{P3} and \cite{DP,DIP}.
Completing this circle of ideas, we note that $d$-move equivalence of links is a special case of congruence modulo $(d,q)$ of links due to
Fox \cite{F1}. Lackenby's study of congruence modulo $(d,q)$ of links lead him to define congruence modulo $d$ of 3-manifolds.

We will give two proofs of the following theorem. The first proof will use Burnside groups and second will  use cohomology ring structure. We let $T^3$ denotes the $3$--torus.

\begin{thm}\label{1} $T^3$ is not weakly $d$--congruent to $\#^3 S^1 \times S^2$ for any $d>2.$ \end{thm}

We remark that we don't know whether or not  the $3$--torus is weakly  $2$--congruent to  $\#^3 S^1 \times S^2$. It seems unlikely. If one could prove that the $3$--torus is not weakly  $2$--congruent to  $\#^3 S^1 \times S^2$, it would provide a second proof of Fox's result \cite{F2} that the $3$--torus is not the double branched
cover of a link. By the trick of Montesinos \cite{M}, the double  branched cover of $S^3$ along a link with $c$ components is weakly $2$--congruent to $\#^{c-1} S^1 \times S^2$.

\begin{proof}[First Proof of Theorem \ref{1}]
The $d$th Burnside group of a group $G$ obtained by quotienting $G$ by the subgroup normally generated by  the $d$th powers of all elements.  The $d$th Burnside group of a manifold $M$  is  the $d$th Burnside group of  the fundamental group of the manifold.  Slightly generalizing an observation of Dabkowski and Przytycki
 \cite[proof of Theorem (1.2)]{DP}, we noted in \cite{G} that the $d$th Burnside group is preserved by weak $d$--congruence.

The $d$th  Burnside group of $T^3$ is abelian. In fact it is 
$\BZ_d^3$.  According to  \cite[Exercise 2.2.19]{MKS}, the $d$th Burnside group of a free group on $r$ generators is nonabelian if $d>2$ and $r>1$. Of course the fundamental group of $\#^3 S^1 \times S^2$  is free on  three generators. \end{proof}

In  \cite[Theorem (2.7)]{G}, we observed that a weak $d$--congruence induces an isomorphism of 
$\BZ_d$--cohomology groups of $M$. Moreover if $d$ is odd, this isomorphism preserves the ring structure. Using Poincare duality, this simply means the  trilinear pairing $t_M$  on $H_1(M,\BZ_d)$ with values in $\BZ_d$ is preserved.

Let $t_M$ denote the trilinear  form  on $H_1(M,\BZ_d)$ with values in $\BZ_{d}$ which  sends $(\chi_1, \chi_2, \chi_3)$ to $(\chi_1\cup \chi_2 \cup  \chi_3) \cap [M]$. If $d$ is even, let $\rho:\BZ_d \rightarrow \BZ_{d/2}$ be reduction modulo ${d/2}$.

\begin{thm}\label{ring} A weak $d$--congruence between $M$ and $M'$ induces an isomorphism $c:H^1(M,\BZ_d) \rightarrow H^1(M',\BZ_d)$.
If $d$ is odd,
 \[  
 t_M(\chi_1, \chi_2, \chi_3)
=  t_{M'}(c(\chi_1), c(\chi_2), c(\chi_3) ).\] 
If 
$d$ is even,  \[\rho  
\left( t_M(\chi_1, \chi_2, \chi_3)\right)
= \rho \left( t_{M'}(c(\chi_1), c(\chi_2), c(\chi_3) )\right).\]
\end{thm}

\begin{proof} The case $d$ odd is \cite[Theorem (2.7)]{G}. The proof, in the case $d$ even, proceeds in exactly the same way.  At the end, we need to see that the triple intersection number $\tau$  
must satisfy $\rho (\tau)= 0\pmod{d/2}.$  This follows from $\tau= -\tau\pmod{d}$,  which holds since the triple intersection number of surfaces is skew symmetric.    \end{proof}

\begin{prop} In the case, $d$ is even, \[  
 t_M(\chi_1, \chi_2, \chi_3)
=  t_{M'}(c(\chi_1), c(\chi_2), c(\chi_3) ) \] need not hold. The $\tau$ that appears in the proof of  Theorem \ref{ring} is congruent to $d/2$ modulo $d.$ \end{prop}

\begin{proof} 
One may  pass from $S^1 \times S^2$ to $L(ds,q)$ by a weak type-d surgery.
Let $\psi$ denote a generator for $H^1(L(ds,q), \BZ_d).$  One has that  $\psi \cup \psi$ is $d/2$ times a generator for  $H^1(L(ds,q), \BZ_d)$
\cite[Example 3.41]{H}. It follows that $t_{L(ds,q)}(\psi,\psi,\psi)= d/2 \pmod{d}.$
On the other hand, $t_{S^1 \times S^2}$ is the zero trilinear form. We note that it follows that $\tau= d/2 \pmod{d}$. \end{proof}

\begin{proof}[Second Proof of Theorem \ref{1}] We   apply Theorem \ref{ring}. If $d$ is odd, we note that $ t_{T^3}$ is non--trivial and  $ t_{\#^3 S^1 \times S^2}$ is zero. If $d$ is even, we observe that  $\rho  \circ t_{T^3}$ is non--trivial and  $\rho \circ t_{\#^3 S^1 \times S^2}$ is zero. 
 \end{proof}

Let $P$ denote the Poincare homology sphere. $P$ can also be described as 
the Brieskorn manifold $\Si(2,3,5)$.  Let $\Si$ denote the  Brieskorn homology sphere $\Si(2,3,7)$. 

\begin{prop}  $P$, $\Si$ and $S^3$ are all  weakly $5$--congruent to each other. However  no two of them are $5$-congruent.\end{prop}

\begin{proof} The last statement is contained in  \cite[Corollary 3.10]{G}. 
By \cite[Lemma (1.1)]{Mi}, $P$ and $\Si$ are double branched covers of $S^3$ along the respectively the $(3,5)$ and $(3,7)$ torus knots.  Both of these knots are  closures of $3$--braids. According to  \cite[Theorem 2.2]{DIP}, the closure of any $3$--braid is $(2,2)$--move equivalent to a trivial link or one of four specified $3$--component links. We have that $(2,2)$ moves are covered in the double branched covers of links by $\pm 2/5$ surgeries \cite{DP, DIP}.  So $P$ and $\Si$  must each be  weakly $5$--congruent to $S^3$ (the double branched cover of the unknot) or  the double branched covers of a link with more than one component.  
But the double branched cover of a $c$--component link will  have first homology with $\BZ_5$ coefficients $\BZ_5^{c-1}$. As this homology group is preserved by
weak $5$--congruence, and both $P$ and $S$ are homology spheres, $P$ and $S$ must be weakly $5$--congruent to $S^3.$
\end{proof}

{We would like to thank Jozef Przytycki for some valuable comments.}

\end{document}